\documentclass[11pt,a4paper]{amsart}
\usepackage{amsmath,amsthm,amssymb,latexsym}
\usepackage{graphicx}

\newcommand{\re}{\text{\rm Re\,}}
\newcommand{\im}{\text{\rm Im\,}}

\newcommand{\bd}{{\mathbb{D}}}
\newcommand{\bn}{{\mathbb{N}}}

\newcommand{\bz}{{\mathbb{Z}}}
\newcommand{\bc}{{\mathbb{C}}}

\newcommand{\cs}{\mathcal{S}}

\newcommand{\cp}{{\mathcal{P}}}

\renewcommand{\a}{\alpha}

\renewcommand{\l}{\lambda}
\renewcommand{\ll}{\Lambda}
\newcommand{\s}{\sigma}

\newcommand{\p}{\varphi}

\renewcommand{\th}{\theta}
\renewcommand{\d}{\delta}
\newcommand{\dd}{\Delta}
\renewcommand{\o}{\omega}

\newcommand{\g}{\gamma}

\newcommand{\ep}{\varepsilon}
\newcommand{\z}{\zeta}

\newcommand{\bsl}{\backslash}
\newcommand{\ovl}{\overline}

\DeclareMathOperator{\diag}{\rm diag}

\DeclareMathOperator{\dist}{\rm dist}

\allowdisplaybreaks
\numberwithin{equation}{section}

\newtheorem{theorem}{Theorem}[section]

\newtheorem{corollary}[theorem]{Corollary}
\newtheorem{proposition}[theorem]{Proposition}

\theoremstyle{definition}

\newtheorem{remark}[theorem]{Remark}

\begin{document}

\title[Jacobi matrices]
{Perturbation determinants and discrete spectra of semi-infinite non-self-adjoint  Jacobi operators}
\author[L. Golinskii]{L. Golinskii}

\address{B. Verkin Institute for Low Temperature Physics and
Engineering, Ukrainian Academy of Sciences, 47 Nauky ave., Kharkiv 61103, Ukraine}
\email{golinskii@ilt.kharkov.ua}

\date{\today}


\maketitle

\begin{abstract}
We study the trace class perturbations of the half-line, discrete Laplacian and obtain a new bound for the
perturbation determinant of the corresponding non-self-adjoint Jacobi operator. Based on this bound, we obtain
the Lieb--Thirring inequalities for such operators. The spectral enclosure for the discrete spectrum and embedded 
eigenvalues are also discussed.
\end{abstract}

\vspace{0.5cm}
\begin{center}
{\large\it In memory of Sergey Naboko (1950--2020)}
\end{center}

\section*{Introduction}

In the last two decades there was a splash of activity around the spectral theory of non-self-adjoint perturbations
of some classical operators of mathematical physics, such as the Laplace and Dirac operators on the whole space,
their fractional powers, and others. Recently, there has been some interest in studying certain discrete models of the above
problem. In particular, the structure of the spectrum for compact, non-self-adjoint perturbations of the free Jacobi
and the discrete Dirac operators has attracted much attention lately. Actually the problem concerns the discrete component
of the spectrum and the rate of its accumulation to the essential spectrum. Such type of results under the various assumptions
on the perturbations are united under a common name {\it Lieb--Thirring inequalities}. 
For a nice account of the existing results on the problem for non-self-adjoint, two-sided Jacobi 
operators, the reader may consult two recent surveys \cite{dhk13} and \cite[Section 5.13]{fran20} and references therein.

The spectral theory of semi-infinite, {\sl self-adjoint} Jacobi matrices is quite popular owing to their tight relation to the
theory of orthogonal polynomials on the real line \cite{sisz}. In contrast, there are only a few papers where semi-infinite, non-self-adjoint
Jacobi matrices are examined \cite{lya68, akwa95, beck01, eggo05, goeg05, goku07, bfv19}. 

The main object under consideration is a semi-infinite Jacobi matrix
\begin{equation}\label{hljac}
J(\{a_j\},\{b_j\},\{c_j\})_{j\in\bn} =
\begin{bmatrix}
b_1 & c_1 &  &  & \\
a_1 & b_2 & c_2 & & \\
& a_2 & b_3 & c_3 & \\
& & \ddots & \ddots & \ddots
\end{bmatrix}, 
\end{equation}
with uniformly bounded complex entries, and $a_jc_j\not=0$, $j\in\bn:=\{1,2,\ldots\}$. The spectral theory of the underlying non-self-adjoint 
Jacobi operator includes, among others, the structure of its spectrum. We denote by $J_0$ the semi-infinite discrete Laplacian, 
i.e., $J_0=J(\{1\},\{0\},\{1\})$. If $J-J_0$ is a compact operator, that is,
$$ \lim_{n\to\infty} a_n=\lim_{n\to\infty} c_n=1,\quad \lim_{n\to\infty} b_n=0, $$
the geometric image of the spectrum is plainly evident
$$ \s(J)=\s_{ess}(J_0)\cup \s_d(J)=[-2,2]\cup \s_d(J), $$
the discrete component $\s_d(J)$, i.e., the set of isolated eigenvalues of finite algebraic multiplicity, is an at most countable 
set of points in $\bc\backslash [-2,2]$ with the only possible limit points on $[-2,2]$. 
To get some quantitative information on the rate of such accumulation one has to impose some further assumptions on the perturbation. Our case of
study in the paper is the {\it trace class} perturbations of the discrete Laplacian $J_0$
\begin{equation}\label{trclper}
J-J_0\in\cs_1 \ \Leftrightarrow \ \sum_{n=1}^\infty (|1-a_n|+|b_n|+|1-c_n|)<\infty,
\end{equation}
see, e.g., \cite[Lemma 2.3]{kisi03}.

Our strategy is similar to that in \cite{eggo05} and \cite{go21-1}. The key issue is the bound for the {\it perturbation determinant}
\begin{equation}\label{perdeter}
L(\l,J):=\det(I+(J-J_0)(J_0-\l)^{-1})
\end{equation}
introduced by M.G. Krein \cite{gokr67} in the late 1950s. The main feature
of this analytic function on the resolvent set $\rho(J_0)=\ovl\bc\bsl [-2,2]$ is that the set of its zeros agrees with the discrete spectrum
of the perturbed operator $J$, and moreover, the multiplicity of each zero equals the algebraic multiplicity of the corresponding eigenvalue.
So the original problem of the spectral theory can be restated as a classical problem of the zero distributions of analytic functions, which goes back 
to Jensen and Blaschke.

The argument is pursued in two steps. The first one results in the bound for the perturbation determinant, typical for the functions
of non-radial growth. The classes of such analytic (and subharmonic) functions in the unit disk were introduced and studied in 
\cite{bgk09, fago09} (for some advances see \cite{bgk18}). The Blaschke-type conditions for the zero sets (Riesz measures) were proved therein, 
with an important amplification in \cite[Theorem 4]{haka11}, better adapted for applications. 
The second step is just the latter result applied to the bound mentioned above.

Such two-step algorithm is applied, by and large, in \cite{haka11} to the two-sided Jacobi operators.
In our approach to the problem the argument in the first step is totally different. The point is that for semi-infinite
Jacobi operators neither the Fourier transform machinery, nor a simple matrix representation for the resolvent of the free Jacobi operator are at 
our disposal. Instead, we deal with the associated three-term recurrence relation
\begin{equation}\label{3term}
u_{k-1}+b_ku_k+a_kc_ku_{k+1}=\l(z)u_k, \quad k\in\bn, \quad \l(z)=z+\frac1z,
\end{equation}
and its solutions. Here $\l(\cdot)$ is the Zhukovsky function which maps the unit disk onto the resolvent set 
$\rho(J_0)=\ovl\bc\bsl [-2,2]$. 

The relation between solutions of \eqref{3term} and eigenvectors of $J$ is straightforward: $(v_k)_{k\ge1}$ is an eigenvector of $J$ if and
only if
$$ (u_k)_{k\ge0}: \ \ u_0=0, \quad u_1=v_1, \quad u_k=\frac{v_k}{a_1\ldots a_{k-1}}\,, \quad k\ge2 $$
is a solution of \eqref{3term}.

The solution $u^+=(u^+_k)_{k\ge0}$ of \eqref{3term} is called the {\it Jost solution} if
\begin{equation}\label{jossol}
\lim_{k\to\infty} z^{-k}u_k^{+}(z)=1, \qquad z\in\bd_0:=\bd\bsl\{0\}.
\end{equation}
Under certain assumptions on $J$ the Jost solution exists and unique (see Theorem \ref{boundjost} below).

We study the Jost solutions by reducing the difference equation \eqref{3term} to a Volterra-type discrete integral equation, the standard idea 
in analysis, see, e.g., \cite[Section 7.5]{te00}, \cite{eggo05, go21-1}. The bounds for the Jost solutions stem from the successive approximations method. 
The crucial point in such an approach is that the perturbation determinant agrees with the first coordinate $u^+_0$ of the Jost solution, 
known as the {\it Jost function}, see the non-self-adjoint version of Killip and Simon \cite[Theorem 2.16]{kisi03} 
(the calculation there has nothing to do with self-adjointness).

The Lieb--Thirring inequality for the discrete spectrum of semi-infinite Jacobi operators looks as follows.

\begin{theorem}\label{mainth}
Let $J-J_0\in\cs_1$. Then for each $\ep\in(0,1)$ there is a constant $C(\ep)>0$ so that
\begin{equation}\label{mainlt}
\sum_{\l\in\s_d(J)} \frac{\dist(\l,[-2,2])}{|\l^2-4|^{{\frac{1-\ep}2}}}\le C(\ep)\dd, \quad \dd:=\sum_{n=1}^\infty (|b_n|+|1-a_nc_n|).
\end{equation}
\end{theorem}

If $J$ is the discrete Schr\"odinger operator, that is, $a_n=c_n\equiv 1$, then
\begin{equation}\label{mainlt1}
\sum_{\l\in\s_d(J)} \frac{\dist(\l,[-2,2])}{|\l^2-4|^{{\frac{1-\ep}2}}}\le C(\ep)\|J-J_0\|_1.
\end{equation}

\begin{remark}\label{simil}
The value $\dd$ in \eqref{mainlt} in place of $\|J-J_0\|_1$ looks quite natural, at least for small perturbations. Indeed, given a 
Jacobi matrix $J$, consider a class $S(J)$ 
of Jacobi matrices
\begin{equation*}
\begin{split}
S(J) &=\{\widehat J:=T^{-1}JT, \ T=\diag(t_j)_{j\in\bn} \ {\rm is \ a \ diagonal \ isomorphism \ of}\ \ell^2(\bn)\}, \\
\widehat J &=J\bigl(\{a_jr_j\}, \{b_j\}, \{c_jr_j^{-1}\}\bigr), \quad r_n=\frac{t_n}{t_{n+1}}, \quad n\in\bn.
\end{split}
\end{equation*} 
Clearly, $\s_d(\widehat J)=\s_d(J)$ since $\widehat J$ is similar to $J$. Hence the left side of \eqref{mainlt} is constant within the class
$S(J)$, and so is $\dd$, in contrast to $\|J-J_0\|_1$. 
For the class $S(J_0)$ both sides of \eqref{mainlt} vanish, whereas $\|J-J_0\|_1$, $J\in S(J_0)$,
can be arbitrarily large.

Next, $|1-a_nc_n|\le |1-a_n|+|1-c_n|+|1-a_n||1-c_n|$, and so
\begin{equation}\label{ddnorm}
\begin{split}
\dd &\le \sum_{n=1}^\infty \Bigl(|b_n|+|1-a_n|+|1-c_n|\Bigr)+\left(\sum_{n=1}^\infty |1-a_n|\right)\left(\sum_{n=1}^\infty |1-c_n|\right) \\
&\le 3\|J-J_0\|_1+\|J-J_0\|_1^2.
\end{split} 
\end{equation}
We see that for small perturbations the value $\dd$ has at least the same order as $\|J-J_0\|_1$.
\end{remark}

The bounds for the Jost functions, obtained in the course of the proof of our main statement, provide some new spectral enclosure results.

\begin{theorem}\label{mainth2}
$(i)$. Let $J-J_0\in\cs_1$. The discrete spectrum $\s_d(J)$ belongs to the following Cassini oval
\begin{equation}\label{casov}
\s_d(J)\subset \left\{\l\in\bc\backslash [-2,2]: \ |\l^2-4|\le \left(\frac{2\dd}{\log 2}\right)^2\right\}. 
\end{equation}

$(ii)$. Assume that
\begin{equation}\label{mom1}
\dd_1:=\sum_{n=1}^\infty n\Bigl(|b_n|+|1-a_{n-1}c_{n-1}|\Bigr)<\infty.
\end{equation}
Then the discrete spectrum is missing, $\s_d(J)=\emptyset$, as long as $\dd_1<\log 2$.
\end{theorem}

Note that the latter effect is inherent to semi-infinite Jacobi matrices.

The Lieb--Thirring inequality for {\sl two-sided} Jacobi operators  is due to Hansmann and Katriel \cite[Theorem 1]{haka11}. It states that
for each $\ep\in(0,1)$ there is a constant $C(\ep)>0$ so that
\begin{equation}\label{hankat}
\sum_{\l\in\s_d(J)} \frac{\dist(\l,[-2,2])^{1+\ep}}{|\l^2-4|^{\frac12+\frac{\ep}4}}\le C(\ep)\|J-J_0\|_1.
\end{equation}
It has been proved recently in \cite{go21-1}, that this bound can be refined. Actually, \eqref{mainlt1} holds for two-sided Jacobi operators
as well. Now the Wronskian of two Jost solutions plays a key role.

The result of Hansmann--Katriel \eqref{hankat} is known to be sharp in the sense that \eqref{hankat} is false for $\ep=0$. 
To prove that, the authors of \cite{bostam20} introduce a special class of two-sided Jacobi operators with rectangular (step) potentials. 
In Section \ref{s3} we deal with an obvious counterpart of this class in semi-infinite setting. 

Given $n\in\bn$ and $h>0$, we study a semi-infinite discrete Schr\"odinger operator with a pure imaginary step potential
\begin{equation}\label{recpot}
J_{n,h}:=
\begin{bmatrix}
b_1 & 1 &  &  & \\
1 & b_2 & 1 & & \\
& 1 & b_3 & 1 & \\
& & \ddots & \ddots & \ddots
\end{bmatrix}, \ 
b_j=\left\{
  \begin{array}{ll}
    ih, & \hbox{$j=1,\ldots,n;$} \\
    0, & \hbox{$j\ge n+1$.}
  \end{array}
\right.
\end{equation}
It is clear that $J_{n,h}=J_0+ih\cp_n$, $\cp_n$ is the orthogonal projection onto the linear span of the first $n$ basis vectors $\{e_1,\ldots,e_n\}$.

We write $J_{n,h}=J_{n,\a}$ for particular values of $h=h_n=n^{-\a}$, $0<\a<~1$.
For such operators some quantitative information about the discrete spectrum, which provides sharpness of Theorem \ref{mainth}, 
can be gathered. 

\begin{theorem}\label{shar}
For the discrete Schr\"odinger operators $J_{n,\a}$ the following lower bound holds for large enough $n$
\begin{equation}\label{sharp}
\frac1{\|J_{n,\a}-J_0\|_1}\,\sum_{\l\in\s_d(J_{n,\a})} \frac{\dist(\l,[-2,2])}{|\l^2-4|^{1/2}}\ge C\log n.
\end{equation}
In particular,
\begin{equation}\label{shar1}
\lim_{n\to\infty} \frac1{\|J_{n,\a}-J_0\|_1}\,\sum_{\l\in\s_d(J_{n,\a})} \frac{\dist(\l,[-2,2])}{|\l^2-4|^{1/2}}=+\infty.
\end{equation}
\end{theorem}

The argument in \cite{bostam20} relies heavily on a simple matrix representation for the resolvent of the whole-line free Jacobi matrix 
and the theory of Kac--Murdock--Szeg\H{o} matrices, neither of which is available in the semi-infinite case. Instead, we analyze directly the
recurrence relations for eigenvectors and apply Rouch\'e's Theorem to the roots of certain algebraic equations.

\section{Jost solutions and discrete Volterra equations}
\label{s1}

We derive the bounds for the Jost solution $u^+$ by reducing the difference equation \eqref{3term} to the
Volterra-type discrete integral equation. The unity of the first coefficient \eqref{3term} appears to be crucial.

Define a (non-symmetric) Green kernel for $k,m\in\bz$ (as a two-sided Laurent matrix) by
\begin{equation}\label{green}
G(k,m;z) :=\left\{
  \begin{array}{ll}
    \frac{z^{m-k}-z^{k-m}}{z-z^{-1}}, & \hbox{$m\ge k,$} \\
    0, & \hbox{$m\le k,$}
  \end{array}
\right. \qquad z\in\bd_0.
\end{equation}
The basic properties of this kernel can be verified directly
\begin{equation}\label{grker}
\begin{split}
G(k,m-1;z)+G(k,m+1;z)-\Bigl(z+\frac1z\Bigr)\,G(k,m;z) &=\d_{k,m}, \\
G(k-1,m;z)+G(k+1,m;z)-\Bigl(z+\frac1z\Bigr)\,G(k,m;z) &=\d_{k,m}.
\end{split}
\end{equation}
We make use of these properties for $k,m\in\bn_0:=\{0,1,\ldots\}$. The kernel
\begin{equation}\label{traker}
T(k,m;z) :=-b_mG(k,m;z)+(1-a_{m-1}c_{m-1})G(k,m-1;z) 
\end{equation}
is a key player of the game. Note that the values of $b_0$, $a_{-1}$, $c_{-1}$ are immaterial.

\begin{theorem}\label{difinteq}
The Jost solution $u^+=(u^+_k)_{k\ge0}$ of the difference equation $\eqref{3term}$ satisfies the discrete Volterra-type equation
\begin{equation}\label{voltr}
u^+_k(z)=z^k+\sum_{m=k+1}^\infty T(k,m;z)u^+_m(z), \quad k\in\bn_0, \quad z\in\bd_0.
\end{equation}
Conversely, each solution $u=(u_k)_{k\ge0}$ of $\eqref{voltr}$ solves $\eqref{3term}$.
\end{theorem}
\begin{proof}
We multiply the first relation \eqref{grker} for $G$ by $u^+_m$, 
relation \eqref{3term} for $u^+$ by $G(k,m)$, and subtract the later from the former
\begin{equation*}
\begin{split}
\Bigl[G(k,m+1)u^+_m-G(k,m)u^+_{m-1}\Bigr] &+\Bigl[-b_mG(k,m)+G(k,m-1)\Bigr]u^+_m \\
&-a_mc_mG(k,m)u^+_{m+1}=\d_{k,m}u^+_m. 
\end{split}
\end{equation*}
Next, we sum up over $m$ from $k+1$ to $N$, taking into account that \newline $G(k,k+1)=1$, $G(k,k)=0$
\begin{equation*}
\begin{split}
G(k,N+1)u^+_N &+\sum_{m=k+1}^N \Bigl[-b_mG(k,m)+G(k,m-1)\Bigr]u^+_m \\
&-\sum_{m=k}^N a_mc_mG(k,m)u^+_{m+1}=u^+_k,
\end{split}
\end{equation*}
or
\begin{equation*}
u^+_k=G(k,N+1)u^+_N-a_Nc_NG(k,N)u^+_{N+1}+\sum_{m=k+1}^N T(k,m)u^+_m.
\end{equation*}
The latter equality holds for arbitrary solutions of \eqref{3term}. If $u^+$ is the Jost solution, then, by \eqref{green} and \eqref{jossol},
$$ \lim_{N\to\infty} \Bigl[G(k,N+1)u^+_N-a_Nc_NG(k,N)u^+_{N+1}\Bigr]=z^k, $$
and \eqref{voltr} follows.

To prove the converse statement, let $u=(u_k)_{k\ge0}$ be any solution of \eqref{voltr}. Then
\begin{equation*}
\begin{split}
u_{k-1}+u_{k+1} &=\Bigl(z+\frac1z\Bigr)\,z^k+T(k-1,k)u_k+T(k-1,k+1)u_{k+1} \\
&+\sum_{m=k+2}^\infty\Bigl[T(k-1,m)+T(k+1,m)\Bigr]u_m.
\end{split}
\end{equation*}
But
\begin{equation*}
\begin{split}
T(k-1,k)u_k &=-b_ku_k, \\
T(k-1,k+1)u_{k+1} &=\Bigl[-b_{k+1}G(k-1,k+1)+(1-a_kc_k)G(k-1,k)\Bigr]u_{k+1} \\
&=-\Bigl(z+\frac1z\Bigr)b_{k+1}u_{k+1}+(1-a_kc_k)u_{k+1} \\
&=\Bigl(z+\frac1z\Bigr)T(k,k+1)u_{k+1}+(1-a_kc_k)u_{k+1}, \\
T(k-1,m)+T(k+1,m) &=\Bigl(z+\frac1z\Bigr)T(k,m).
\end{split}
\end{equation*}
Finally,
\begin{equation*} 
u_{k-1}+u_{k+1}=-b_ku_k+(1-a_kc_k)u_{k+1}+\Bigl(z+\frac1z\Bigr)\,\Bigl(z^k+\sum_{m=k+1}^\infty T(k,m)u_m\Bigr),
\end{equation*}
which is \eqref{3term}. The proof is complete.
\end{proof}

The further study of the Volterra equation \eqref{voltr} relies upon the modified kernel
\begin{equation}\label{newker}
\widetilde T(k,m;z):=T(k,m;z)\,z^{m-k}.
\end{equation}
It is easy to verify that $\widetilde T(k,m;\cdot)$ are polynomials of $z$. Indeed,
\begin{equation*}
\begin{split}
z^{m-k}G(k,m;z) &=z^{m-k}\,\frac{z^{m-k}-z^{k-m}}{z-z^{-1}}=z\,\frac{z^{2(m-k)}-1}{z^2-1}\,, \quad m\ge k; \\
z^{m-k}G(k,m-1;z) &=z^{m-k}\,\frac{z^{m-k-1}-z^{k-m+1}}{z-z^{-1}}=z^2\,\frac{z^{2(m-k-1)}-1}{z^2-1}\,, \quad m\ge k+1, 
\end{split}
\end{equation*}
as claimed. The bounds for the kernel $\widetilde T$ follow from the latter relations
\begin{equation*}
\begin{split}
\bigl|z^{m-k}G(k,m;z)\bigr| &\le |z|\cdot\min\Bigl\{(m-k)_+, \ \frac2{|1-z^2|}\Bigr\}, \quad (a)_+:=\max(a,0); \\
\bigl|z^{m-k}G(k,m-1;z)\bigr| &\le |z|^2\cdot\min\Bigl\{(m-k-1)_+, \ \frac2{|1-z^2|}\Bigr\},
\end{split}
\end{equation*}
and so
\begin{equation}\label{bounker}
\begin{split}
\bigl|\widetilde T(k,m;z)\bigr| &\le \d_m |z|\min\Bigl\{(m-k)_+, \ \frac2{|1-z^2|}\Bigr\}, \quad z\in\ovl{\bd}, \\
\d_m &:=|b_m|+|1-a_{m-1}c_{m-1}|, \qquad a_0=c_0=1.
\end{split}
\end{equation}
In particular,
\begin{equation}\label{bounker1}
\bigl|\widetilde T(k,m;z)\bigr|\le |\o(z)|\d_m, \quad \o(z):=\frac{2z}{1-z^2}, \quad z\in\bd_1:=\ovl\bd\bsl\{\pm1\}.
\end{equation}

\begin{theorem}\label{boundjost}
$(i)$. Assume that
\begin{equation}\label{ell1}
\dd=\sum_{n=1}^\infty \d_n<\infty.
\end{equation}
Then the equation \eqref{voltr} has a unique solution $u^+=(u^+_k)_{k\ge0}$ so that $u^+_k$ are analytic in $\bd$, continuous in $\bd_1$, and
for $k\in\bn_0$ and $z\in\bd_1$
\begin{equation}\label{bouell}
|u^+_k(z)-z^k|\le |z|^{k}\left(e^{|\o(z)|s_0(k)}-1\right), \quad s_0(k):=\sum_{n=k+1}^\infty \d_n.
\end{equation}

$(ii)$. Assume that $\eqref{mom1}$ holds
\begin{equation*}
\dd_1=\sum_{n=1}^\infty n\d_n<\infty.
\end{equation*}
Then \eqref{voltr} has a unique solution $u^+=(u^+_k)_{k\ge0}$ so that $u^+_k$ are analytic in $\bd$, continuous in $\ovl\bd$, and
for $k\in\bn_0$ and $z\in\ovl\bd$
\begin{equation}\label{boumom}
|u^+_k(z)-z^k|\le |z|^{k}\left(e^{|z|s_1(k)}-1\right), \quad s_1(k):=\sum_{n=k+1}^\infty n\d_n.
\end{equation}
\end{theorem}
\begin{proof}
It is advisable to introduce new variables in \eqref{voltr}
$$ f_j(z):=u^+_j(z)z^{-j}-1, \qquad j\in\bn_0, $$
so the equation \eqref{voltr} turns into
\begin{equation}\label{voltr1}
\begin{split}
f_k(z) &=g_k(z)+\sum_{m=k+1}^\infty \widetilde T(k,m;z)f_m(z), \\ 
g_k(z) &:=\sum_{m=k+1}^\infty \widetilde T(k,m;z), \quad k\in\bn_0, \quad z\in\bd_1.
\end{split}
\end{equation}
It is clear from \eqref{bounker1} and the assumption \eqref{ell1} that the latter series converges absolutely and uniformly on compact
subsets of $\bd_1$ and so represents an analytic function in $\bd$. Moreover,
$$ |g_k(z)|\le |\o(z)|s_0(k), \qquad k\in\bn_0, \quad z\in\bd_1. $$

We are going to solve \eqref{voltr1} by using the successive approximations method. Let
$$ f_{k,1}(z):=g_k(z), \quad f_{k,j+1}(z):=\sum_{m=k+1}^\infty \widetilde T(k,m;z)f_{m,j}(z), \quad j\in\bn. $$
It is easy to see, by induction, that
\begin{equation}\label{sucap1}
|f_{k,p}(z)|\le \frac{\bigl(|\o(z)|s_0(k)\bigr)^p}{p!}\,, \qquad p\in\bn. 
\end{equation}
Indeed, once the bound is true for $p=1$, we assume that it holds for $p=1,2,\ldots,j$ and $k\in\bn_0$. Then
$$ |f_{k,j+1}(z)|\le|\o(z)|\,\sum_{m=k+1}^\infty \d_m|f_{m,j}(z)|\le\frac{|\o(z)|^{j+1}}{j!}\,\sum_{m=k+1}^\infty \d_m s_0^j(m). $$
The elementary inequality $(a+b)^{j+1}-a^{j+1}\ge(j+1)ba^j$, $a,b\ge0$ gives with $a=s_0(m)$, $b=\d_m$, $a+b=s_0(m-1)$
$$ \sum_{m=k+1}^\infty \d_m s_0^j(m)\le \frac1{j+1}\,\sum_{m=k+1}^\infty \bigl(s_0^{j+1}(m-1)-s_0^{j+1}(m)\bigr)=\frac{s_0^{j+1}(k)}{j+1}\,, $$
and so
$$ |f_{k,j+1}(z)|\le \frac{|\o(z)|^{j+1} s_0^{j+1}(k)}{j!(j+1)}=\frac{\bigl(|\o(z)|s_0(k)\bigr)^{j+1}}{(j+1)!}\,, $$
as claimed.

Hence, the series
\begin{equation*}
f_k:=\sum_{j=1}^\infty f_{k,j}(z)
\end{equation*}
converges absolutely and uniformly on compact subsets of $\bd_1$ and represents an analytic in $\bd$ function, continuous in $\bd_1$.
It satisfies \eqref{voltr1}
\begin{equation*}
\begin{split}
f_k(z)-g_k(z) &=f_k(z)-f_{k,1}(z)=\sum_{j=1}^\infty f_{k,j+1}(z)=\sum_{j=1}^\infty \sum_{m=k+1}^\infty \widetilde T(k,m;z)f_{m,j}(z) \\
&=\sum_{m=k+1}^\infty \widetilde T(k,m;z)f_{m}(z).
\end{split}
\end{equation*}
This solution admits the bound, see \eqref{sucap1},
\begin{equation*}
|f_k(z)|\le\sum_{j=1}^\infty |f_{k,j}(z)|\le \sum_{j=1}^\infty \frac{\bigl(|\o(z)|s_0(k)\bigr)^j}{j!}=e^{|\o(z)|s_0(k)}-1,
\end{equation*}
which is \eqref{bouell}.

As far as uniqueness goes, suppose that there are two solutions of \eqref{voltr}
$$ u^+=(u^+_k)_{k\ge0}, \qquad v^+=(v^+_k)_{k\ge0}. $$ 
Assume further that $z\not=0$. By \eqref{bounker1}, we have
\begin{equation}\label{uniq1} 
|u^+_k(z)-v^+_k(z)|\le|\o(z)|\,\sum_{m\ge k+1}\d_m|u^+_m(z)-v^+_m(z)|=:q_k(z). 
\end{equation}
Clearly, $q_k\searrow 0$ as $k\to\infty$.

We fix $z$ and assume first that $q_p>0$ for all $p\in\bn_0$. By \eqref{uniq1},
$$ \frac{q_{p-1}-q_p}{q_p}=\frac{\d_p|u^+_p(z)-v^+_p(z)||\o(z)|}{q_p}\le |\o(z)|\d_p, \quad q_p\le q_N\,\prod_{j=p+1}^N\bigl(1+|\o(z)|\d_j\bigr). $$
The latter product converges as $N\to\infty$, so $q_p=0$ in contradiction with our assumption.

Next, let $l\in\bn_0$ exist so that $q_l=0$. Then, in view of monotonicity, $q_{l+1}=q_{l+2}=\ldots=0$, and, in the opposite way, successively,
$$ |u^+_l-v^+_l|=0 \ \Rightarrow \ q_{l-1}=0 \ \Rightarrow \ |u^+_{l-1}-v^+_{l-1}|=0 \ \Rightarrow \ q_{l-2}=0 \ldots \ \Rightarrow \ |u^+_{0}-v^+_{0}|=0, $$
so the uniqueness follows.

(ii). The proof goes along the same line of reasoning with the auxiliary bounds for $z\in\ovl\bd$
$$ \Bigl|\widetilde T(k,m;z)\Bigr|\le |z|m\d_m, \quad |g_k(z)|\le |z|s_1(k), \quad |f_{k,p}(z)|\le \frac{\bigl(|z|s_1(k)\bigr)^p}{p!}\,. $$
The proof is complete.
\end{proof}

\section{Perturbation determinant and the Lieb--Thirring inequality}
\label{s2}

Under our main assumption \eqref{trclper}, the perturbation determinant $L(\l,J)$ \eqref{perdeter}  is a well-defined analytic function on $\bd$.

\begin{theorem}
Let $J-J_0\in\cs_1$. Then the bound holds
\begin{equation}\label{bo1}
\log|L(\l(z),J)|\le |\o(z)|\dd=\frac{2|z|\dd}{|1-z||1+z|}, \qquad z\in\bd.
\end{equation}
Under assumption $\eqref{mom1}$,
\begin{equation}\label{bo2}
\log|L(\l(z),J)|\le|z|\dd_1, \qquad z\in\ovl\bd.
\end{equation}
\end{theorem}
\begin{proof}
By the non-self-adjoint version of \cite[Section 2]{kisi03} (the calculation there is algebraic and so immediately extends to the non-self-djoint case),
the Jost solution $u^+$ of \eqref{3term} equals
$$ u^+_k(z)=z^k L(\l(z),J^{(k)}), \qquad k\in\bn_0, $$
where $J^{(k)}$, the $k$-stripped Jacobi matrix, is obtained from $J$ by dropping the first $k$ rows and columns. So, \eqref{bo1} and \eqref{bo2}
follow directly from \eqref{bouell} and \eqref{boumom}, respectively, with $k=0$.
\end{proof}

We are now ready for

{\it Proof of Theorem \ref{mainth}}. 

According to \cite[Theorem 4]{haka11}, for each $\ep\in(0,1)$ 
there is a constant $C(\ep)>0$ so that the Blaschke-type condition holds for the zero set (divisor) 
$Z(L)$, $L$ in \eqref{bo1}, $L(0)=1$ $$ \sum_{\z\in Z(L)} (1-|\z|)\frac{|\z^2-1|^\ep}{|\z|^\ep}\le C(\ep)\dd, $$
(each zero is taken with its multiplicity). The latter inequality turns into \eqref{mainlt}, when we go over to the Zhukovsky images,
taking into account the distortion for the Zhukovsky function \cite[Lemma 7]{haka11}
\begin{equation*}
\frac12\,\frac{|1-z^2|(1-|z|)}{|z|}\le \dist(\l,[-2,2])\le \frac{1+\sqrt2}2\,\frac{|1-z^2|(1-|z|)}{|z|}. 
\end{equation*} 

For the discrete Schr\"odinger operators $(a_n=c_n\equiv 1)$, one has 
$$ \dd=\sum_{n=1}^\infty |b_n|=\|J-J_0\|_1, $$ 
and \eqref{mainlt1} follows. The proof is complete.

In view of \eqref{ddnorm}, it might be worth comparing the key inequality \eqref{bo1} with
$$ \log|L(z)|\le \frac{C_{abs}}{|1-z^2|^2}\,(\|J-J_0\|_1+\|J-J_0\|_1^2), \quad z\in\bd, $$
the result obtained in \cite[Theorem 2.8]{kisi03} for the self-adjoint case.

There is yet another consequence of Theorem \ref{boundjost}, (i), which concerns embedded eigenvalues of the Jacobi operator $J$. The result
is likely to be known (cf. \cite{ibst19} for two-sided discrete Schr\"odinger operators), so we briefly outline its proof.

\begin{corollary}
Let $J-J_0\in\cs_1$. Then the operator $J$ has no embedded eigenvalues, i.e., eigenvalues on $(-2,2)$.
\end{corollary}
\begin{proof}
Assume on the contrary, that $\l=2\cos\th$, $0<\th<\pi$, is the eigenvalue of $J$. In this case $\l$ is also the eigenvalue for the modified
Jacobi operator $\widehat J:=J(\{1\},\{b_j\},\{a_jc_j\})$, which is similar to $J$, see Remark \ref{simil}. 
Denote by $h=(h_k)_{k\ge1}$ the corresponding eigenvector, so we have an $\ell^2$-solution
$h'=(0,h)$ of \eqref{3term} with $z=e^{i\th}$. 

On the other hand, the Jost solution is known to be continuous in $\bd_1$, so the second solution $u^+=(u^+_k)_{k\ge0}$ of \eqref{3term} comes up.
It is clear from \eqref{bouell} that
\begin{equation}\label{joscir}
|u^+_k(e^{i\th})|=1+o(1), \qquad k\to\infty.
\end{equation}

The Wronskian of these two solutions
$$ W(u^+,h')=\prod_{j=k+1}^\infty (a_jc_j)^{-1}\,(u^+_kh_{k+1}-u^+_{k+1}h_k) $$
at the point $e^{i\th}$ is $k$-independent, and as $h\in\ell^2$, and $u^+$ is bounded, we see that $W(u^+,h')\equiv 0$. Hence,
$u^+$ and $h$ are to be linearly dependent, that contradicts \eqref{joscir} and $h\in\ell^2$.
\end{proof}

{\it Proof of Theorem \ref{mainth2}}.

The bound \eqref{bouell} with $k=0$ means that $L\not=0$ in $\bd$ as long as 
\begin{equation*} 
|\o(z)|\dd<\log 2, \qquad \frac{|z|}{|1-z^2|}<\frac{\log 2}{2\dd}\,. 
\end{equation*}
Since
$$ |\l^2-4|=\left|\frac{1-z^2}{z}\right|^2, $$
we come to \eqref{casov}.

Under assumption $\dd_1<\log 2$, the bound \eqref{boumom} with $k=0$ implies $L\not=0$ on $\bd$, so
$\s_d(J)$ is missing, as claimed. The proof is complete.

\begin{remark}
The spectral enclosures are normally derived from the Birman--Schwinger principle. Precisely,
$$ \l(z)\in\s_d(J) \ \Rightarrow \ \|K(z)\|\le1, $$
$K$ is the Birman--Schwinger operator. In our case such approach leads to the inclusion
\begin{equation*}
\s_d(J)\subset \bigl\{\l\in\bc\backslash [-2,2]: \ |\l^2-4|\le 36^2\|J-J_0\|_1^2\bigr\}. 
\end{equation*}

The approach based on the Jost functions for semi-infinite discrete Schr\"odinger operators gives
\begin{equation*}
\s_d(J)\subset \Bigl\{\l\in\bc\backslash [-2,2]: \ |\l^2-4|\le \frac4{(\log 2)^2}\,\|J-J_0\|_1^2\Bigr\}. 
\end{equation*}

Note that for two-sided discrete Schr\"odinger operators the sharp oval which contains the discrete spectrum is known \cite{ibst19}
\begin{equation*}
\s_d(J)\subset \bigl\{\l\in\bc\backslash [-2,2]: \ |\l^2-4|\le \|J-J_0\|_1^2\bigr\}. 
\end{equation*}
\end{remark}

\section{Semi-infinite discrete Schr\"odinger operators with rectangular potentials}
\label{s3}

We begin with a slight refinement of \cite[Lemma 4]{bostam20}.

\begin{proposition}\label{locspec}
The discrete spectrum of the operator $J_{n,h}$ $\eqref{recpot}$ is contained in the rectangle
\begin{equation}\label{losp}
\s_d(J_{n,h})\subset \{\l\in\bc: \ |\re\l|<2, \ 0<\im\l<h\}.
\end{equation}
\end{proposition}
\begin{proof}
If $J_{n,h}g=\l g$, $\|g\|=1$, then
$$ \l=\langle J_{n,h}g, g\rangle=\langle J_0 g,g\rangle+ih\,\langle\cp_n g,g\rangle. $$
With $g=(g_k)_{k\ge1}$, $g_0:=0$, we have
\begin{equation*}
\begin{split}
|\re\l| &=\bigl|\langle J_0 g,g\rangle\bigr|=\Bigl|\sum_{j\ge1} (g_{j-1}+g_{j+1})\ovl{g_j}\Bigr|=2\Bigl|\sum_{j\ge1} \re(g_{j-1}\ovl{g_j})\Bigr| \\
&\le 2\sum_{j\ge1} |g_{j-1}g_j|\le 2\sum_{j\ge1}|g_j|^2=2.
\end{split}
\end{equation*}
Actually, the latter inequality is strict or, otherwise, $2|g_{j-1}g_j|=|g_{j-1}|^2+|g_j|^2$ would imply $|g_{j-1}|=|g_j|$ for all $j$ 
that is impossible for a nonzero vector $g\in\ell^2$. So, $|\re\l|<2$, as claimed.

Next, obviously $0\le\im\l=h\langle\cp_n g,g\rangle\le h$. Moreover,
$$ \im\l=0 \ \Rightarrow \ \langle\cp_n g,g\rangle=0 \ \Rightarrow \ \cp_n g=0, $$
so $J_{n,h}g=J_0g=\l g$, and $g$ is the eigenvector of $J_0$ that is also impossible. Similarly,
$$ \im\l=h \ \Rightarrow \ \langle\cp_n g,g\rangle=\|\cp_n g\|^2=1 \ \Rightarrow \ \cp_n g=g, $$
and again, the conclusion $J_0 g=(\l-ih)g$ is false.
\end{proof}

Let 
\begin{equation*}
J_{n,h}g=\l g=\Bigl(z+\frac1z\Bigr)g, \quad z\in\bd_0=\bd\backslash\{0\}; \quad g=(g_j)_{j\ge1}\in\ell^2(\bn), \quad g\not=0. 
\end{equation*}
The system of recurrence relations for the coordinates looks as follows
\begin{equation}\label{eigrec}
\begin{split}
ih g_1+g_2 &=\l g_1, \\
g_{j-1}+ih g_j+g_{j+1} &=\l g_j, \quad j=2,\ldots,n, \\
g_{j-1}+g_{j+1} &=\l g_j, \quad j=n+1,\ldots\,.
\end{split}
\end{equation}

We begin with the characteristic equation for the second series of relations
\begin{equation}\label{char}
t^2-(\l-ih)t+1=(t-t_1)(t-t_2)=0, \quad t_j=t_j(z,h), \ \ j=1,2.
\end{equation}
Furthermore, 
$$ t_1t_2=1, \quad t_1+t_2=t_1+t_1^{-1}=\l-ih, $$ 
and $\im\l<h$ implies $0<|t_1|<1<|t_2|$. Put $\z:=t_1$, so
\begin{equation}\label{specvar}
\z+\frac1{\z}=\l-ih=z+\frac1z-ih, \qquad \z,z\in\bd_0.
\end{equation}

Yet another relation between $z$ and $\z$ comes from the adjustment in \eqref{eigrec}. Indeed,
\begin{equation*}
\begin{split}
g_j &=a\z^j+b\z^{-j}, \quad j=1,\ldots,n+1, \quad |a|+|b|>0; \\
g_j &=g_nz^{j-n}, \quad j=n,n+1,\ldots,
\end{split}
\end{equation*}
since $g\in\ell^2$. For $j=2$ we have $g_2=a\z^2+b\z^{-2}$. On the other hand, the first relation in \eqref{eigrec} gives
$$ g_2=(\l-ih)g_1=(\z+\z^{-1})(a\z+b\z^{-1})=a\z^2+b\z^{-2}+a+b \ \Rightarrow \ a+b=0. $$
For $j=n+1$
$$ g_{n+1}=a\z^{n+1}+b\z^{-n-1}=(a\z^n+b\z^{-n})z \ \Rightarrow \ a\z^n(z-\z)+b\z^{-n}(z-\z^{-1})=0. $$

The system of two homogeneous, linear equations has a nontrivial solution if and only if its determinant vanishes
\begin{equation}\label{determ}
\begin{vmatrix}
1 & 1 \\
\z^{n}(z-\z) & \z^{-n}(z-\z^{-1})
\end{vmatrix}
=\z^{-n}(z-\z^{-1})-\z^{n}(z-\z)=0.
\end{equation}
Hence, if $\l\in\s_d(J_{n,h})$ then \eqref{specvar} and \eqref{determ} hold.

Conversely, let a pair $z,\z\in\bd_0$ satisfy \eqref{specvar} -- \eqref{determ}, so $\z$ is the root of the characteristic equation 
\eqref{char} in $\bd$. It is a matter of direct calculation to check that
$g=(g_j)_{j\ge1}$ with
$$ g_j:=\z^j-\z^{-j}, \quad j=1,2,\ldots,n+1; \quad g_j=g_nz^{j-n}, \quad j=n+1,\ldots $$
is the eigenvector of $J_{n,h}$ with the eigenvalue $\l$.

Let us express $z$ from \eqref{determ}
\begin{equation}\label{spvind} 
z=\frac{\z^{n+1}-\z^{-n-1}}{\z^n-\z^{-n}}=\frac{\z^{2n+2}-1}{\z^{2n+1}-\z}\in\bd 
\end{equation}
and plug this expression into \eqref{specvar}. We come to the main algebraic equation associated with $J_{n,h}$
$$ \z+\frac1{\z}+ih=\frac{\z^{2n+2}-1}{\z^{2n+1}-\z}+\frac{\z^{2n+1}-\z}{\z^{2n+2}-1}\,, $$
or
\begin{equation}\label{aleq}
P_n(\z):=\z^{2n-1}(\z^2-1)^2-ih(1-\z^{2n})(1-\z^{2n+2})=0.
\end{equation}

Conversely, let $\z$ be a root of \eqref{aleq} in $\bd$ so that \eqref{spvind} holds. Then \eqref{determ} is true, and \eqref{specvar} is
a simple consequence of the equality
$$ (\z^{2n+2}-1)^2+(\z^{2n+1}-\z)^2-(\z+\z^{-1})(\z^{2n+2}-1)(\z^{2n+1}-\z)=\z^{2n}(\z^2-1)^2. $$

Thereby, we come to the following result.

\begin{proposition}\label{algeq}
Let $z\in\bd_0$. The number $\l=z+z^{-1}\in\s_d(J_{n,h})$ if and only if there is a root $\z\in\bd_0$ of the polynomial $P_n$ $\eqref{aleq}$ so that
$\eqref{spvind}$ holds.
\end{proposition}

Our next goal is to show that the number of such roots is large enough (proportional to $n$).

From now on we put $h:=n^{-\a}$, $0<\a<1$ is fixed, and write $J_{n,h}=J_{n,\a}$. The parameter $n$ is assumed to be large.

Take $0<c_2<\a<c_1<1$ and put
\begin{equation}\label{rad}
r_j =r_j(n):=1-c_j\,\frac{\log n}{2n+1}\,, \quad j=1,2, \quad r_1<r_2.
\end{equation}
Clearly,
\begin{equation}\label{rad1}
r_j^{2n+1}=n^{-c_j}\left(1+O\left(\frac{\log^2 n}{n}\right)\right), \quad j=1,2.
\end{equation}

\begin{proposition}\label{segm}
The circular segment
$$ S=S(n,a):=\Bigl\{w=re^{i\p}: \ r_1<r<r_2, \ a\,\frac{\pi}2<\p<\Bigl(\frac12-a\Bigr)\,\frac{\pi}2\Bigr\}, \ \ 0<a<\frac14, $$
contains at least $N=(1/4-a)n+O(1)$ simple roots of $P_n$ $\eqref{aleq}$.
\end{proposition}
\begin{proof}
Write $P_n=P_{n,2}-P_{n,1}$ with
\begin{equation*}
\begin{split}
P_{n,1}(w) &=2w^{2n+1}+in^{-\a}, \\
P_{n,2}(w) &=w^{2n+1}(w^2+w^{-2})+in^{-\a}(w^{2n}+w^{2n+2}-w^{4n+2}). 
\end{split}
\end{equation*}
We wish to compare the roots of $P_n$ with the roots of a simple binomial $P_{n,1}$. The latter are known explicitly
\begin{equation*}
\begin{split}
P_{n,1}(w) &=0 \ \Leftrightarrow \ w=w_{n,k}=\g_{n}\exp\Bigl\{\frac{i\pi}{2(2n+1)}\,(4k-1)\Bigr\}, \quad k=1,\ldots,2n+1,\\
\g_n &=|w_{n,k}|=\Bigl(\frac{n^{-\a}}2\Bigr)^{1/2n+1}=1-\a\,\frac{\log n}{2n+1}+O\Bigl(\frac1n\Bigr).
\end{split}
\end{equation*}

Define the values $\p_k=\p_k(n)$ by
\begin{equation}\label{bouseg} 
\p_k:=\frac{\pi(4k+1)}{2(2n+1)}  \ \Rightarrow \ \p_{k-1}<\arg w_{n,k}<\p_k, \ \ k=1,\ldots,2n+1, 
\end{equation}
and the segments $S_k=S_k(n)$ by
$$ S_k:=\bigl\{w=re^{i\p}: \ r_1<r<r_2, \ \p_{k-1}<\p<\p_k\bigr\}. $$
Since $r_1(n)<\g_n<r_2(n)$ for large $n$, each segment $S_k$ contains exactly one root $w_{n,k}$ of $P_{n,1}$.
As it turns out, $P_{n,1}$ dominates $P_{n,2}$ on the boundary $\partial S_k$ for certain values of $k$, specified below.

We have  
\begin{equation*}
P_{n,1}(re^{i\p}) =2r^{2n+1}\cos (2n+1)\p+i(2r^{2n+1}\sin (2n+1)\p+n^{-\a}),
\end{equation*}
and so
\begin{equation}\label{e1}
|P_{n,1}(re^{i\p})| \ge |2r^{2n+1}\sin (2n+1)\p+n^{-\a}|. 
\end{equation}
For the second polynomials 
$$ P_{n,2}(re^{i\p})=r^{2n+1}e^{i(2n+1)\p}\bigl(r^2e^{2i\p}+r^{-2}e^{-2i\p}+in^{-\a}(re^{i\p}+r^{-1}e^{-i\p}-r^{2n+1}e^{i(2n+1)\p})\bigr), $$
and since 
\begin{equation}\label{e3}
\begin{split}
r^2e^{2i\p}+r^{-2}e^{-2i\p} &=(r^2-1)e^{2i\p}+2\cos 2\p+(r^{-2}-1)e^{-2i\p} \\
&= 2\cos 2\p+O\left(\frac{\log n}{n}\right), \quad r_1\le r\le r_2,
\end{split}
\end{equation}
we come to the bound
\begin{equation}\label{e2}
|P_{n,2}(re^{i\p})|\le r^{2n+1}\bigl(2\,|\cos2\p|+O(n^{-\a})\bigr), \quad r_1\le r\le r_2,
\end{equation}
with $O$ uniform in $r$ and $\p$.

Let first $\p=\p_k$, $r_1\le r\le r_2$. By \eqref{bouseg}, $\sin(2n+1)\p_k=1$ for all $k$, so \eqref{e1} , \eqref{e2} imply
\begin{equation*}
\begin{split}
|P_{n,1}(re^{i\p_k})| &\ge 2r^{2n+1}+n^{-\a}, \\
|P_{n,2}(re^{i\p_k})| &\le 2r^{2n+1}|\cos 2\p_k|+o(n^{-\a}),
\end{split}
\end{equation*}
(see \eqref{rad1} for the last term). Hence, $|P_{n,1}(re^{i\p_k})|>|P_{n,2}(re^{i\p_k})|$ for all $k$, that is, 
the domination holds on the radial sides of $S_k$.

Next, let $r=r_2$, $\p_{k-1}\le\p\le\p_k$, be the exterior arc of $S_k$. By \eqref{rad1},
\begin{equation*}
\begin{split}
|P_{n,1}(re^{i\p})| &=|2r^{2n+1}e^{i(2n+1)\p}+in^{-\a}|= r^{2n+1}\Bigl(2+ O\bigl(n^{-(\a-c_2)}\bigr)\Bigr), \\
|P_{n,2}(re^{i\p_k})| &\le r^{2n+1}\bigl(2|\cos 2\p_k|+O(n^{-\a}\bigr)).
\end{split}
\end{equation*}
We choose $k$ in such a way that the value $|\cos 2\p_k|$ is separated from $1$.
An elementary calculation shows that for
\begin{equation}\label{rangk}
\begin{split}
k\in \ll &=\ll_{n,a}:=[a_n', a_n''], \\
a_n' &=a\,\frac{2n+1}4 +1, \qquad a_n''=\Bigl(\frac12-a\Bigr)\,\frac{2n+1}4-1,
\end{split}
\end{equation}
the inequality holds
\begin{equation}\label{rangp}
a\pi<2\p_{k-1}<2\p_k<\Bigl(\frac12-a\Bigr)\pi.
\end{equation}
So, again 
$$ |P_{n,1}(r_2e^{i\p})|>|P_{n,2}(r_2e^{i\p})|, \qquad \p_{k-1}\le\p\le\p_k, \quad k\in \ll. $$

Finally, along the interior arc of $S_k$, that is, $r=r_1$, $\p_{k-1}\le\p\le\p_k$, we have $r_1^{2n+1}=O(n^{-c_1})=o(n^{-\a})$, and so
$$ |P_{n,1}(r_1e^{i\p})|\ge n^{-\a}, \qquad |P_{n,2}(r_1e^{i\p})|=O(n^{-c_1}), $$
as needed.

By Rouch\'e's Theorem, each region $S_k$ with $k\in\ll$ contains one simple root of $P_n$. The number of such regions is 
$N=(1/4-a)n+O(1)$, and, by \eqref{rangp}, all of them lie in $S$. The proof is complete.
\end{proof}

In the rest of the paper we focus on the roots of $P_n$ just found. Denote them by $\z_{k}=\z_k(n)=\rho_ke^{i\th_k}$,
$\rho_{k}=\rho_k(n)$, $\th_{k}=\th_k(n)$. The equality \eqref{aleq} with $\z=\z_k$ reads
$$ -in^{\a}\z_{k}^{2n+1}\bigl(\z_{k}-\z_{k}^{-1}\bigr)^2=(1-\z^{2n}_{k})(1-\z^{2n+2}_{k}). $$
As in \eqref{e3} above,
$$ \bigl(\z_{k}-\z_{k}^{-1}\bigr)^2=-4\sin^2\th_{k}\left(1+O\Bigl(\frac{\log n}{n}\Bigr)\right). $$
Since $\rho_{k}^{2n+1}=O\bigl(n^{-c_2}\bigr)$, we end up with
\begin{equation}\label{e8}
4i\sin^2\th_{k}\,n^\a\rho_{k}^{2n+1}\,e^{i(2n+1)\th_{k}}=1+O\bigl(n^{-c_2}\bigr).
\end{equation}

Take the absolute value and the real part in \eqref{e8}:
\begin{equation}\label{e9}
\begin{split}
4\sin^2\th_{k}\,n^\a\rho_{k}^{2n+1} &=1+O\bigl(n^{-c_2}\bigr), \\
-4\sin^2\th_{k}\,n^\a\rho_{k}^{2n+1}\sin(2n+1)\th_k &=1+O\bigl(n^{-c_2}\bigr),
\end{split}
\end{equation}
with $O$ uniform in $k\in\ll$. By plugging the first equality in \eqref{e9} into the second one, we have
\begin{equation}\label{e9.1} 
\sin (2n+1)\th_{k}=-1+O\bigl(n^{-c_2}\bigr). 
\end{equation}
Regarding $\th_k$ themselves, it is clear from the construction of $S_k$ that
\begin{equation}\label{e9.2}
\th_k=\frac{\pi}{2(2n+1)}\,(4k+1)+O\Bigl(\frac1n\Bigr), \quad k\in\ll.
\end{equation}

For the range of arguments in $S$ we have
\begin{equation}\label{bounds} 
0<4\sin^2 a\,\frac{\pi}2\le 4\sin^2\th_{k}\le 4\sin^2\Bigl(\frac12-a\Bigr)\,\frac{\pi}2<2, 
\end{equation}
so it follows from \eqref{e9} that
\begin{equation}\label{e12}
\rho_{k}=1-\a\,\frac{\log n}{2n+1}+O\Bigl(\frac1n\Bigr),
\end{equation}
with $O$ uniform in $k\in\ll$ (cf. with $\g_{n}=|w_{n,k}|$). 

\smallskip

{\it Proof of Theorem \ref{shar}}.

In view of Proposition \ref{algeq}, to make sure that the roots $\{\z_k\}_{k\in\ll}$ generate eigenvalues of $J_{n,\a}$, 
we have to check the relation \eqref{spvind}. Indeed, let
$$ z_k=z_k(n):=\frac{\z_k^{2n+2}-1}{\z_k^{2n+1}-\z_k}\,. $$
Then
\begin{equation*}
1-|z_{k}|^2 =\frac{(\rho_k^2-1)(1-\rho_k^{4n+2})-4\rho_k^{2n+2}\sin (2n+1)\th_k\sin\th_k}{|\z_k^{2n+1}-\z_k|^2}\,.
\end{equation*}
By \eqref{e12}, the first term in the numerator is 
$$ (\rho_{k}^2-1)(1-\rho_{k}^{4n+2})=O\Bigl(\frac{\log n}{n}\Bigr). $$
As for the second one, the first equality in \eqref{e9} and \eqref{bounds} imply
$$ \rho_{k}^{2n+1}=\frac{n^{-\a}}{4\sin^2\th_{k}}\,(1+O(n^{-c_2}))>\frac{n^{-\a}}4\,, $$
and in view of \eqref{e9.1} and \eqref{bounds}, the second term is
$$ -4\sin (2n+1)\th_{k}\sin\th_{k}>c>0. $$
Hence, $1-|z_{k}|^2>0$ for large enough $n$, as claimed.

Therefore, the numbers
$$ \l_k=\l_k(n):=\z_k+\z_k^{-1}+in^{-\a}=\rho_ke^{i\th_k}+\rho_k^{-1}e^{-i\th_k}+in^{-\a} $$
are {\sl distinct} eigenvalues of $J_{n,\a}$ for all large enough $n$ and $k\in\ll$ \eqref{rangk}.
We have, as in \eqref{e3} above,
\begin{equation}\label{eige}
\l_k=2\cos\th_k+in^{-\a}+O\Bigl(\frac{\log n}{n}\Bigr),
\end{equation}
and so, by Proposition \ref{locspec},
\begin{equation}\label{e15}
{\rm dist}\,(\l_{k},[-2,2]) =\im\l_{n,k}=n^{-\a}(1+o(1)). 
\end{equation}

As we restrict ourselves with the eigenvalues $\l_k$, $k\in\ll$, 
\begin{equation*}
\sum_{\l\in\s_d(J_{n,\a})} \frac{\dist(\l,[-2,2])}{|\l^2-4|^{1/2}} \ge \sum_{k\in\ll} \frac{\dist(\l_k,[-2,2])}{|\l_k^2-4|^{1/2}}=:I_{n,a}.
\end{equation*}
Next, by the location of the discrete spectrum and \eqref{eige}, we see that
\begin{equation}\label{e15.1}
\begin{split}
|\l_k+2|^{1/2} &\le (4+n^{-\a})^{1/2}<\sqrt5, \\
|\l_k-2|^{1/2} &\le 2\sin\frac{\theta_{k}}2+n^{-\a/2}(1+o(1))\le \theta_{k}+n^{-\a/2}(1+o(1)),
\end{split}
\end{equation}
and so \eqref{e15} and \eqref{e15.1} imply (see also \eqref{e9.2})
\begin{equation*}
\begin{split}
I_{n,a} &\ge \frac{n^{-\a}(1+o(1))}{\sqrt5}\,\sum_{k\in\ll}\Bigl[\frac{\pi}{2(2n+1)}(4k+1)+n^{-\a/2}(1+o(1))\Bigr]^{-1} \\
&\ge \frac{n^{1-\a}(1+o(1))}{\pi\sqrt5}\,\sum_{k\in\ll}\frac1{k+u_n}\,, \quad u_n:=\frac{n^{1-\a/2}}{\pi}\,(1+o(1))+1.
\end{split} 
\end{equation*}

An elementary bound
$$ \sum_{m_1}^{m_2} \frac1{k+u_n}>\int_{m_1}^{m_2} \frac{dx}{x+u_n}=\log\frac{m_2+u_n}{m_1+u_n} $$
with $m_1:=[a_n']+1$, $m_2:=[a_n'']$ (see \eqref{rangk})
gives
\begin{equation}\label{ina} 
I_{n,a}\ge \frac{n^{1-\a}(1+o(1))}{\pi\sqrt5}\,\log\frac{1-2a+4n^{-\a/2}(1+o(1))}{a+2n^{-\a/2}(1+o(1))}\,.
\end{equation}

Up to this point we have tacitly assumed that the value $a\in (0, 1/4)$ from Proposition \ref{segm} is fixed. As a matter of fact, $a$ can vary with $n$.
For example, if $a=n^{-\a/4}$, \eqref{ina} yields
$$ I_{n,a}\ge Cn^{1-\a}\,\log n. $$
The bound \eqref{sharp} now follows from
$$ \|J_{n,\a}-J_0\|_1=\sum_{j=1}^n |b_j(n)|=n^{1-\a}. $$
The proof is complete.

\begin{remark}
The Jost solution $u^+=(u_j^+)_{j\ge0}$ for Jacobi operators with the step potential can be found explicitly. Indeed,
the recurrence relation  
\begin{equation*}
\begin{split}
u_{k-1}^+(z) +ihu_k^+(z) +u_{k+1}^+(z) &=\Bigl(z+\frac1z\Bigr)\,u_k^+(z), \quad k=1,2,\ldots,n; \\ 
u_k^+(z) &=z^k, \quad k=n,n+1\ldots 
\end{split}
\end{equation*}
can be resolved, and we come to the following expression for the Jost function $u_0^+$ (which is the same as the perturbation determinant)
$$ L(z,J_{n,h})=z^n U_n\Bigl(\frac{z+z^{-1}-ih}2\Bigr)-z^{n+1}U_{n-1}\Bigl(\frac{z+z^{-1}-ih}2\Bigr), $$
$U_k$ is the Chebyshev polynomial of the second kind. It might be a challenging problem analyzing the roots of the polynomial on the right side directly.
\end{remark}

\end{document}